\documentclass[12pt]{amsart}
\usepackage[dvipsnames]{xcolor}
\usepackage{graphicx} 
\usepackage{babel}
\usepackage{float}
\usepackage{amsmath,amssymb,amsfonts,mathrsfs,mathtools}
\usepackage[shortlabels]{enumitem}
\setlist{leftmargin=6mm}
\usepackage[centering]{geometry}
\usepackage{soul}
\usepackage{cancel}
\usepackage{tikz,tikz-network}
\usepackage{overpic}
\usepackage{transparent}
\usepackage[style=alphabetic, backend=biber]{biblatex}
\usepackage[pdftex,  colorlinks=true]{hyperref} 
    \hypersetup{urlcolor=RoyalBlue, linkcolor=RoyalBlue,  citecolor=black}
\usepackage[capitalize]{cleveref} 

\theoremstyle{plain}
\newtheorem{theorem}{Theorem}[section]
\newtheorem{proposition}[theorem]{Proposition}
\newtheorem{cor}[theorem]{Corollary}
\newtheorem{lemma}[theorem]{Lemma}

\theoremstyle{definition}

\newtheorem{definition}[theorem]{Definition}
\newtheorem{example}[theorem]{Example}

\newcommand{\Map}{\operatorname{Map}}

\newcounter{commentcounter}

\title[Limiting laminations]{Limiting laminations and boundary weights for infinite-type mapping classes}
\author{Carolyn Abbott AND Phuong Pham}

\begin{document}

\maketitle

\begin{abstract}
    For every infinite-type surface $\Sigma$ with an isolated puncture that admits a shift map and every natural number $n$, we study an intrinsically infinite-type mapping class $g_n$ and a simple closed curve whose orbit under $g_n$ splits into $n$ subsequences, each of which converges in the coarse Chabauty topology to a geodesic lamination on $\Sigma$. These laminations are cyclically permuted by $g_n$, and their union is a $g_n$--invariant geodesic lamination. We further show the attracting limit point of $g_n$ in the boundary of the relative arc graph has weight $n$. In particular, intrinsically infinite-type loxodromic isometries of the relative arc graph realize every finite weight.
\end{abstract}

\section{Introduction}
    For a finite-type surface, pseudo-Anosov mapping classes have two important properties.  They act loxodromically on the curve graph, and the asymptotic behavior of the positive and negative iterates of a simple closed curve is governed by the stable and unstable geodesic laminations, respectively.  Thus, in the finite-type setting, these two phenomena occur together. For infinite-type surfaces, it is natural to ask to what extent this remains true.

   In this paper, we focus on mapping classes that are loxodromic isometries of hyperbolic graphs associated to $\Sigma$.  While there are several natural candidates for such a graph \cite{B-NV, DFV},  we restrict our attention to the \textit{relative arc graph}, introduced by Aramayona, Fossas, and Parlier \cite{AFP}, which can be associated to any infinite-type surface that has an isolated puncture; see Section~\ref{sec:relarcgraph}.  In \cite{Abbott_2022, Abbott_2025}, the first author, Miller, and Patel construct two infinite families of loxodromic isometries of the relative arc graph consisting of \textit{intrinsically infinite-type} elements of $\Map(\Sigma)$, that is, mapping classes that are not contained in the closure of the compactly supported elements $\overline{\Map_c(\Sigma)}$. 
    For the first family, a similar property as for pseudo-Anosov mapping classes holds: there exists a simple closed curve on $\Sigma$ whose images under positive powers of an element of the family converge to a geodesic lamination on $\Sigma$ \cite[Theorem~1.2]{Abbott_2022}.  
    
    For the second family, the orbit of the curve does not converge.  Instead, it splits into $n$ subsequences, one for each residue class modulo $n$, each of which converges to a geodesic lamination.

    \begin{theorem}\label{main}
        Let $\Sigma$ be an infinite-type surface with an isolated puncture that admits a shift map, and equip $\Sigma$ with its conformal hyperbolic metric so that it is equal to its convex core.  For each $n\in \mathbb N$ there is an element $g_n\in \Map(\Sigma)$ and a simple closed curve $c$ on $\Sigma$ such that $(g_n^{nk+t}(c))_{k\in \mathbb N}$ converges in the coarse Chabauty topology to a geodesic lamination $L_t$ as $k\rightarrow\infty$ for each $t\in\{0,1,\cdots,n-1\}$.  
    \end{theorem}

    The laminations obtained in Theorem~\ref{main} are pairwise non-transverse, and so their union forms a single geodesic lamination.  The element $g_n$ cyclically permutes the sublaminations $L_t$, and so their union is invariant under $g_n$.

\begin{theorem}\label{thm:main3}
    The union $L=\bigcup_{t=0}^{n-1}L_t$ is a $g_n$--invariant geodesic lamination.   Moreover, $g_n(L_t)=L_{t+1}$, with indices taken modulo $n$.
\end{theorem}

As an application, we obtain intrinsically infinite-type loxodromic isometries whose attracting limit points realize every finite weight. Bavard--Walker describe the boundary of the relative arc graph as a space of ``cliques of high-filling rays;" see Section~\ref{sec:relarcgraph} for a detailed description.  For now, it suffices to say that each point in $\partial A(\Sigma,p)$ corresponds to a clique of a certain kind of ray on $\Sigma$, and the cardinality of each clique is called the \textit{weight} of the boundary point. Bavard and Walker ask for examples of ``fundamentally infinite-type" loxodromic isometries whose limit points have weight greater than one \cite{BavardWalker}.  Associated to the $n$ laminations $L_t$ are $n$ rays based at the isolated puncture $p$. We show that these rays form the clique representing the attracting fixed point of $g_n$ in the boundary of the relative arc graph.

    \begin{theorem}\label{thm:main2}
        For every $n\geq 2$, there exists an intrinsically infinite-type mapping class $g_n\in \operatorname{MCG}(\Sigma)$ that is a loxodromic isometry of the relative arc graph $\mathcal A(\Sigma,p)$ whose attracting limit point in $\partial \mathcal A(\Sigma, p)$ has weight $n$.
    \end{theorem}

    Previously, boundary points of weight greater than 1 were known only for loxodromics lying in the closure of the compactly supported mapping classes, including finite-type pseudo-Anosovs and the examples of Morales–Valdez \cite{MoralesValdez}. By contrast, the first known intrinsically infinite-type loxodromics have weight-one limit points\footnote{While this is not written explicitly in that paper, it is clear from the lamination constructed in \cite[Section~8]{Abbott_2025}; see Section~\ref{sec:main} for further discussion.}. Thus Theorem~\ref{thm:main2} gives the first intrinsically infinite-type loxodromics with higher-weight limit points, and the construction realizes every finite weight. \\

    The mapping classes $g_n$ in the statements of the theorems are  compositions of a standard shift map on $\Sigma$ and a Dehn twist about a curve enclosing the isolated puncture and  $n$ ``subsurfaces being shifted;" see Section~\ref{sec:main} for a precise definition. Unlike a pseudo-Anosov mapping class, $g_n$ may fix simple closed curves, so the convergence in Theorem~\ref{main} is necessarily a statement about a specially chosen curve $c$, rather than every simple closed curve. 
  
    The proof of Theorem \ref{main} uses the coding scheme introduced in \cite{Abbott_2025} and work from \cite{Abbott_2022} on understanding cancellation by homotopy in the images $g^n(c)$.  Although there is substantial cancellation between successive iterates, long initial subsegments recur after $n$ iterates, and the length of these common segments increases with $k$.  We convert the relevant arcs to closed curves and apply \u{S}ari\'{c}'s train-track criterion, described in Section~\ref{sec:laminations}, to obtain convergence of each residue class modulo $n$ to a geodesic in the universal cover.  From this we extract coarse Chabauty convergence to a lamination on $\Sigma$.

   The proof of Theorem~\ref{thm:main3} uses a more precise description of these limits. Each $L_t$ is the closure of a distinguished geodesic $\gamma^t$, and, after modifying $\gamma^t$ near the isolated puncture, we obtain a ray $\beta^t$. The rays $\beta^t$ are shown to form the clique corresponding to the attracting fixed point of $g_n$; this simultaneously gives Theorem~\ref{thm:main2} and implies that the laminations $L_t$ have no transverse intersections. The cyclic action of $g_n$ on the rays then gives the $g_n$--invariance in Theorem~\ref{thm:main3}.

\subsection*{Acknowledgements:}  The first author was partially supported by NSF grant DMS-2340341.

\subsection*{AI statement:} The authors used ChatGPT Sol 5.6 as an auxiliary tool for proofreading, improving exposition, and informal proof-checking, including identifying points in arguments that required clarification and suggesting possible revisions.  All mathematical arguments and final text were independently checked by the authors, and the authors take full responsibility for the paper's content.

\section{Background}
    
\subsection{Classification of surfaces}

A surface $S$ is of \emph{finite-type} if its fundamental group is finitely generated and is of \emph{infinite-type} otherwise. 

   Let $\Sigma$ be a surface with possibly non-empty boundary, and let $\operatorname{Homeo^+}(\Sigma,\partial \Sigma)$ be the group of orientation-preserving homeomorphisms $\Sigma\rightarrow \Sigma$ that restrict to the identity on the boundary. We endow $\operatorname{Homeo}^+(\Sigma,\partial \Sigma)$ with the \emph{compact-open topology}, which is generated by the subbasis of open sets of the form 
   \[
   U_{K,V}=\{f\in\operatorname{Homeo^+}(\Sigma,\partial \Sigma)\,|\, f(K)\subset V\}
   \]where $K\subset S$ is compact and $V\subset S$ is open.
    \begin{definition}
        Let $X$ be a surface and $\operatorname{Homeo_0}(X,\partial X)$ be the path-connected component containing the identity of $\operatorname{Homeo^+}(X,\partial X)$. The \emph{mapping class group} of $X$ is the group 
        \[\Map(X)=\operatorname{Homeo^+}(X,\partial X)/\operatorname{Homeo}_0(X,\partial X)\]
        equipped with the quotient topology of the compact-open topology on $\operatorname{Homeo^+}(X,\partial X)$. 
    \end{definition}

Let $\Map_c(\Sigma)$ denote the subgroup of mapping classes with compact support.  Following \cite{Abbott_2022}, we say $\phi\in \Map(\Sigma)$ is \textit{intrisically infinite type} if it is not contained in the closure of $\Map_c(\Sigma)$. One important class of intrinsically infinite-type mapping class is shift maps.

\begin{definition}\label{shift}
    Let $S$ be $\mathbb{R}\times [-1,1]$ with open disks centered at $(n,0)$ for $n\in\mathbb{Z}$ with radius $\frac{1}{4}$ removed. For each such disk, attach homeomorphic copies of a topologically non-trivial surface with one boundary component to its boundary, forming a new surface $S'$. A \emph{shift} on $S'$ is the map that acts like translation on $S'$, moving $(x,y)$ to $(x+n,y)$ for $n\in\mathbb{Z}$ and tapering to the identity on $\partial S'$.

    If $\Sigma$ is a surface containing an embedded copy of $S'$, then a shift on $S'$ can be extended via the identity on $\Sigma-S'$ to a \textit{shift map} on $\Sigma$. The \emph{standard shift map} on $\Sigma$ is when $n=1$. We say a surface $\Sigma$ \textit{admits a shift map} if it contains an embedded copy of $S'$.  See Figure~\ref{fig:StandardShift}.
\end{definition}

When the surface with one boundary component glued to $S$ is a one-holed torus, then the shift map is called a \emph{handle shift}, which was first defined by Patel and Vlamis \cite{PatelValmis:2018}.

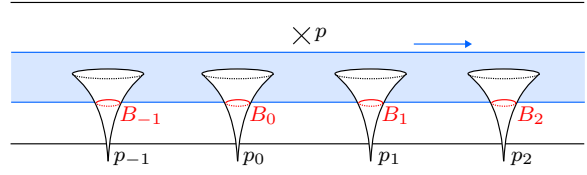
\begin{figure}
    \centering
\begin{overpic}[width=3in]{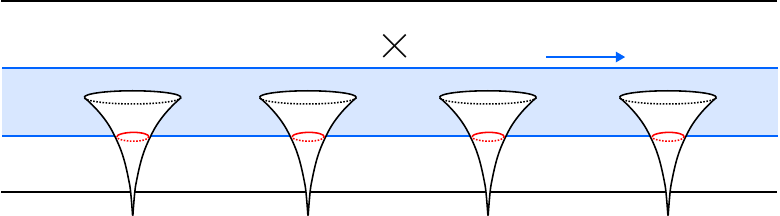}
\put(64, 0){\tiny$p_1$}
\put(87, 0){\tiny$p_2$}
\put(18, 0){\tiny$p_{-1}$}
\put(40.5, 0){\tiny$p_{0}$}
\put(53, 21.5){\tiny$p$}
\put(19, 7){\textcolor{red}{\tiny$B_{-1}$}}
\put(41.5, 7){\textcolor{red}{\tiny$B_{0}$}}
\put(64.5, 7){\textcolor{red}{\tiny$B_{1}$}}
\put(88, 7){\textcolor{red}{\tiny$B_{2}$}}
    \end{overpic}
    \caption{A biinfinite flute surface with the domain of a shift map overlayed on top.  The isolated puncture is labeled $p$.  The $B_i$ are the boundary components of the surface $S$ from Definition~\ref{shift}, and the once-punctured disks with punctures labeled $p_i$ are the topologically non-trivial surfaces.  The  shift sends each puncture in the shaded blue region to the puncture to its right and tapers to the identity map on the boundary of this region.}
    \label{fig:StandardShift}
\end{figure}

\subsection{Laminations}\label{sec:laminations}

In this subsection, we discuss laminations on infinite-type surfaces.  For a comprehensive treatment of the subject, see \cite{Saric:2021}.
    
    \begin{definition}
        A \emph{geodesic lamination} $\lambda$ on a hyperbolic surface $\Sigma$ is a union of disjoint embedded geodesics that form a closed subset of $\Sigma$. The embedded geodesics are called \emph{leaves} of the lamination.  
    \end{definition}

    Let $\Sigma$ be a connected compact orientable surface. A \emph{pants decompostion} of $\Sigma$, denoted $\{P_n\}$, is a collection of pairwise disjoint simple closed curves $\alpha_i$ on $\Sigma$ such that $\Sigma-\cup_{i}\alpha_i$ is a union of surfaces $P_n$,  each homeomorphic to a sphere with three boundary components.  A \emph{train track} on $\Sigma$ is an embedded smooth $1$-complex $\tau$ such that every vertex is trivalent and at each vertex all edges have the same tangent line. 
    Train tracks can be used to study laminations on both finite-type and infinite-type surfaces. 
    
    Fix an infinite-type surface $\Sigma$ and a locally finite geodesic pants decomposition $\{P_n\}$. \u{S}ari\'{c} constructs a particular train track $\Theta$ in \cite{Saric:2021} that will be useful here; see Figure \ref{fig:pantsdecomp}. Consider the lift $\Tilde{\Theta}$ of $\Theta$ to the universal cover $\Tilde{\Sigma}$ of $\Sigma$. A \emph{bi-infinite edge path} of $\Tilde{\Theta}$ is a bi-infinite sequence of edges in $\Tilde{\Theta}$ meeting smoothly at each vertex. By \cite[Proposition 4.6]{Saric:2021}, every bi-infinite edge path has exactly two accumulation points in the boundary at infinity $\partial_\infty\Tilde{X}$. Therefore, the endpoints of a bi-infinite edge path $\Tilde{\alpha}$ of $\Tilde{\Theta}$ determine a unique geodesic $G(\Tilde{\alpha})$ in $\Tilde{\Sigma}$. Given a component $\Tilde{\alpha}$ of the lift of an edge path $\alpha$, we denote by $G(\alpha)$ the projection of $G(\Tilde{\alpha})$ to $\Sigma$. Note that $G(\alpha)$ does not depend on the choice of component $\Tilde{\alpha}$.  
    \begin{definition}
        \begin{itemize}
            \item[(1)] A geodesic $g$ of $\Tilde{\Sigma}$ is said to be \emph{weakly carried by $\Tilde{\Theta}$} if $G(\Tilde{\alpha})=g$ for some bi-infinite path $\Tilde{\alpha}$ in $\Tilde{\Theta}$.
            \item[(2)] $G(\alpha)$ is said to be \emph{weakly carried by $\Theta$} if $G(\Tilde{\alpha})$ is weakly carried by $\Tilde{\Theta}$.
            \item[(3)] A geodesic lamination $L$ on $\Sigma$ is said to be \emph{weakly carried} by $\Theta$ if every leaf of $\ell\in L$ is weakly carried by $\Theta$. 
        \end{itemize}
    \end{definition}
    
    The following result from \cite{Saric:2021} characterizes geodesic laminations in terms of bi-infinite edge paths. 

        \begin{proposition} [{\cite[Proposition 4.11]{Saric:2021}}] \label{prop:biinfedgepaths}
        The set of geodesic laminations on $\Sigma$ that are weakly carried by $\Theta$ is in one-to-one correspondence with the families $\Gamma$ of bi-infinite edge paths of $\Tilde{\Theta}$ that satisfy:
        \begin{itemize}
            \item[(i)] any two bi-infinite edge paths $\alpha$ and $\alpha'$ in $\Gamma$ do not cross; and
            \item[(ii)] if $\alpha$ is a bi-infinite edge path such that for any finite edge subpath there is a bi-infinite edge path in $\Gamma$ that contains it, then $\alpha\in \Gamma$.
        \end{itemize}
    \end{proposition}

    Finally, \u{S}ari\'{c} also describes  when a sequence of geodesics converges to a lamination.
    
    \begin{proposition} [{\cite[Proposition 4.9]{Saric:2021}}] \label{prop:convtogeod}
        Let $g_n,g$ be geodesics in $\Tilde{\Sigma}$ weakly carried by a train track $\Tilde{\Theta}$. Denote by $\Tilde{\alpha}_n,\Tilde{\alpha}$ the corresponding bi-infinite edge paths in $\Tilde{\Theta}$. Then $g_n$ converges to $g$ as $n\rightarrow\infty$ if and only if for each finite subpath $\Tilde{\alpha}'$ of $\Tilde{\alpha}$ there exists an $n_0\geq 0$ such that $\Tilde{\alpha}'$ is contained in $\Tilde{\alpha}_n$ for all $n\geq n_0$. 
    \end{proposition}

One technicality that will be important for this paper is that in the discussion surrounding \cite[Proposition 4.9]{Saric:2021}, \u{S}ari\'{c} assumes that the geodesics do not run out a cusp at either end; in particular, the geodesics in \cite{Saric:2021} do not start or end at a puncture on $\Sigma$.

We consider the space of geodesic laminations with the \textit{coarse Chabauty topology} \cite{CEG06}; see \cite{ABNV} for a discussion of this topology in the context of infinite-type surfaces. As we will not need the full definition, we simply define convergence in this topology.

\begin{lemma}
    A sequence $L_i$ of geodesic laminations converges to a geodesic lamination $L$ in the coarse Chabauty topology if for every leaf $\ell$ of $L$, there is a sequence of leaves $\ell_i\in L_i$ such that $\ell_i\to \ell$ as geodesics in $\mathbb H^2$.
\end{lemma}

\subsection{Codes of arc}
    Let $\Sigma$ be a surface that admits a shift map and has an isolated puncture $p$. The \emph{front} of $\Sigma$ is an open connected neighborhood of an embedded copy of $S'$ as in Definition \ref{shift}, which we can assume without loss of generality contains $p$. The complement of the front of $\Sigma$ is called the \emph{back} of $\Sigma$. We label the punctures in $S'$ by integers $\{k_i\}_i$. Let $\alpha$ be an oriented arc on the front of $\Sigma$ that starts and ends at $p$. In \cite[Section 3.1]{Abbott_2025}, it was shown how to assign a code to $\alpha$ that completely determines $\alpha$ up to homotopy. Intuitively, the characters of the code are determined by following the path of $\alpha$ from its start to its end, both of which are at $p$. The code contains the character $i_o$ if $\alpha$ passes over the corresponding puncture $k_i$ and $i_u$ if $\alpha$ passes under $k_i$. If $\alpha$ does not pass over or under $k_i$, it does not contain either $i_o$ or $i_u$. When it is unimportant whether $\alpha$ passes over or under the puncture $k_i$, we use the character $i_{o/u}$. An example of the code is below.

    \begin{example}\label{ex:code} The code for the arc in  Figure \ref{fig:code_eg} is 
    \[P_s0_o\cdots (-n+3)_o(-n+3)_u \cdots0_u1_u2_u3_u3_o2_u1_u0_u\cdots(-n+3)_u(-n+3)_o\cdots0_oP_s\]
        \begin{figure}[H]
            \centering
            \def\svgwidth{4in}
            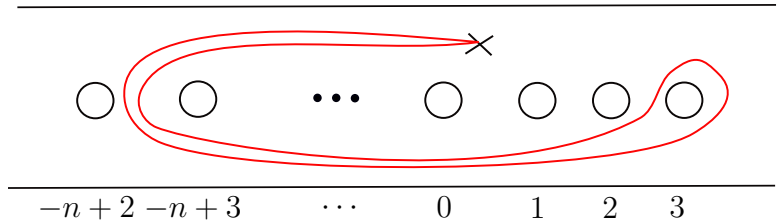
            \caption{The arc in Example~\ref{ex:code}.}
            \label{fig:code_eg}
        \end{figure}
    \end{example}

    A code is \textit{reduced} if it does not contain the same  character twice in a row.
    Given an arc $\alpha$, we let $\ell(\alpha)$ denote the number of characters in a reduced code for $\alpha$. We refer to $\ell(\alpha)$ as the \textit{code length} of $\alpha$.  

\subsection{The boundary of the relative arc graph}\label{sec:relarcgraph}
    The \textit{relative arc graph} of an infinite-type surface $\Sigma$ with an isolated puncture $p$, denoted $\mathcal A(\Sigma,p)$, is graph whose vertices are isotopy classes of simple arcs starting and ending at $p$, with edges corresponding to disjointness.  This graph was defined by Aramayona, Fossas, and Parlier, and is an infinite-diameter hyperbolic graph on which $\Map(\Sigma,p)$, the subgroup fo $\Map(\Sigma)$ fixing $p$, acts by isometries \cite{AFP}.  

    Bavard and Walker  show that the boundary of $\mathcal A(\Sigma,p)$ is homeomorphic to the space of cliques of \textit{high-filling rays} \cite{BavardWalker}.  A ray is \textit{1-filling} if it is disjoint from a loop, that is, from a ray that starts and ends at $p$.  A ray is \textit{2-filling} if it intersects every loop and is disjoint from a 1-filling ray.  Finally, a ray is \textit{high-filling} if it is not a loop and is neither 1- or 2-filling \cite[Lemma~5.6.4]{BavardWalker}.  The number of high-filling rays in the clique is called the \textit{weight} of the boundary point.

    It follows from work in \cite{AFP} that every compactly supported pseudo-Anosov mapping class is a loxodromic isometry of $\mathcal A(\Sigma,p)$, that is, it acts as translation along a quasi-geodesic axis.  To determine the weight of the limit point of such a homeomorphism $\phi$, one considers the foliation preserved by $\phi$.  The weight of the limit point of $\phi$ is equal to the number of prongs in the foliation at the singularity $p$.   Intuitively, each prong corresponds to a high-filling ray starting at $p$, entering a path between leaves of the foliation at the singularity and then following the foliation.  The high-filling rays constructed in this manner are pairwise disjoint, and hence form a clique. 
    
\section{The limiting lamination}\label{sec:main}
    Following the setup in \cite{Abbott_2025}, we fix a surface $\Sigma$ that admits a shift map $h$ and has an isolated puncture $p$ fixed by $h$. The shift induces an embedding of the surface $S$ into $\Sigma$, where $S$ is as in Definition~\ref{shift}, and restricts to a shift on $S$, which we also call $h$.  By construction, $S$ has a countable collection of compact boundary components, which we label $B_i$ for $i\in \mathbb Z$,  where $h(B_i)=B_{i+1}$. For the remainder of the paper, fix $n\geq 2$, and let $\chi=\chi_n$ be the blue curve in Figure \ref{fig:ChiShift}.  Note that $\chi$ encloses the $n\geq 2$ boundary components $B_{-n+2}, B_{-n+1},\dots, B_0,B_1$. The restriction of the domain of $h$ to $S$ is shown in  Figure~\ref{fig:ChiShift}.
    
     \begin{figure}[H]
         \centering
         \begin{overpic}[width=.7\linewidth]{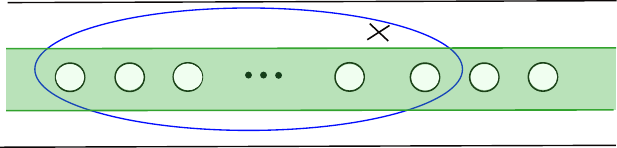}
         \put(10,14){\footnotesize $B_{n-2}$}
         \put(20,14){\footnotesize $B_{n-1}$}
         \put(30,14){\footnotesize $B_{n-3}$}
         \put(51,14){\footnotesize $B_{0}$}
         \put(63,14){\footnotesize $B_{1}$}
         \put(73,14){\footnotesize $B_{2}$}
         \put(83,14){\footnotesize $B_{3}$}
         \put(15,2){{\color{blue}$\chi$}}
         \put(62,21){$p$}
         \end{overpic}
         \caption{The simple closed curve $\chi$, in blue, and the surface $S$, in green.}
         \label{fig:ChiShift}
     \end{figure}
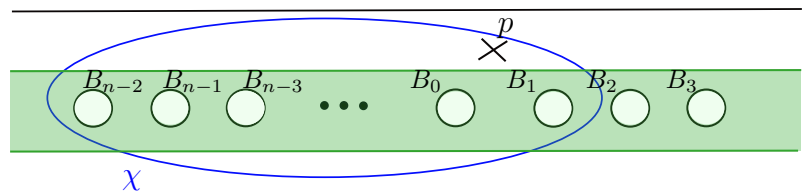
    
    \noindent  Let 
    \[g=g_n= h\circ T_{\chi}\in \Map(\Sigma).\]
    We emphasize that $\chi$ and $g$ both depend on $n$; for simplicity, and since $n$ is fixed, we will suppress the dependence in our notation.

    \begin{figure}[H]
         \centering
         \includegraphics[width=.7\linewidth]{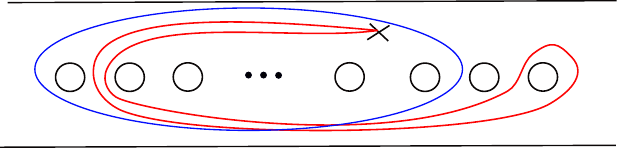}
         \caption{The simple closed curve $\chi$, in blue, and the arc $\alpha_0$, in red.}
         \label{fig:Chialpha0}
     \end{figure}

    Fix an arc $\alpha_0$, as in Figure~\ref{fig:Chialpha0}, and let $\alpha_i=g^{i}(\alpha_0)$ for $i\geq 0$.  Our goal is to use the reduced code for each $\alpha_i$ to find the code for its image under $g$.  While it is possible to find the image of each character in the code one by one and then concatenate the result, this code is rarely reduced.  The problem is that we are considering homotopy classes of arcs; finding a reduced code corresponds to pulling the arc tight. From the point of view of the code, in the unreduced code for the image of an arc, there may be cancellation among various strings of characters. For example, the image of $P_s1_o2_u2_0...$ is 
    \[P_s1_o1_oP_o0_o\cdots(-n+3)_o(-n+3)_u\cdots3_u3_o...,\]
    which is homotopic to $P_s0_o\cdots(-n+3)_o(-n+3)_u\cdots3_u3_o$; see Figure \ref{fig:cancellation}, where the unreduced image is dotted and the reduced solid.  

    \begin{figure}[H]
        \centering
        \includegraphics[width=.9\linewidth]{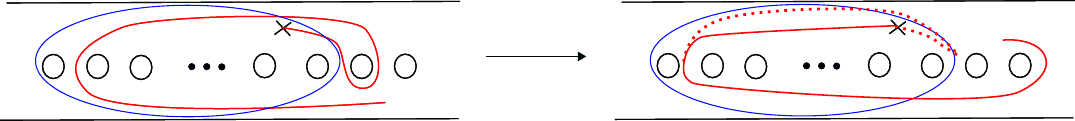}
        \caption{Cancellation via homotopy.  The homotopy class of curve on the right is the image of (an initial portion of) the curve on the left under $g$.  The code for the dotted arc on the right is unreduced, while the code for the solid arc on the right is reduced.}
        \label{fig:cancellation}
    \end{figure}

    How this cancellation can occur, as well as which characters prevent or ``block" cancellation, was investigated in depth in \cite{Abbott_2022,Abbott_2025}. The following lemma from \cite{Abbott_2022} is useful for direct computation of reduced codes of images of curves under $g$.  

    \begin{lemma}[{\cite[Lemma 3.1.3]{Abbott_2022}}]\label{block_AMP}
        If $\delta$ is an arc on $\Sigma$ containing the characters $k_{o/u}$ for any $[2,\infty)$, then there is no cancellation between the images under $g$ of the portion of $\delta$ before $k_{o/u}$ and the portion after $k_{o/u}$.
    \end{lemma}

    Given a pair $k_{o/u}k_{u/o}$ with $k\geq 2$ and applying Lemma \ref{block_AMP} first to $k_{o/u}$ and then to $k_{u/o}$, we obtain the following corollary.

    \begin{cor}\label{block cancellation}
        For $k\geq 2$, we have
        \[g(k_{o/u}k_{u/o})=(k+1)_{o/u}(k+1)_{u/o},\]
        and there is no cancellation between the images under $g$ of the portion of $\delta$ before $k_{o/u}k_{u/o}$ and the portion after $k_{o/u}k_{u/o}$.
    \end{cor}
    \noindent  Following the terminology in \cite{Abbott_2022,Abbott_2025}, we say that a pair $k_{o/u}k_{u/o}$ with $k\geq 2$, \emph{blocks cancellation}. 
     
     In light of this corollary, we define the \emph{initial segment} of an arc $\alpha_i$ to be the sub-segment of the code of $\alpha_i$ beginning with $P_s$, ending with a pair $k_{o/u}k_{u/o}$  with $k\geq 2$ and containing exactly one copy of $1_{o/u}2_{o/u}$. Essentially, this is the part of the code for $\alpha_i$ that starts at the puncture and ends at the first pair that blocks cancellation. For example, the initial segment of $\alpha_0$, which will play an important role in this paper, is 
     \[
     \delta_0:=P_s0_o\cdots (-n+3)_o(-n+3)_u \cdots0_u1_u2_u3_u3_o
     \] when $n\geq 3$.  When $n=2$, the initial segment is $\delta_0=P_s1_u2_u3_u3_o$.  The top left of Figure \ref{fig:ini} shows $\delta_0$ in red as a subpath of $\alpha_0$ for $n\geq 3$. Note that the arc $\alpha_0$ is symmetric about the final $3_u3_o$ pair in $\delta_0$.  In particular, $\ell(\delta_0)=\frac12\ell(\alpha_0)+1$.
     For $i\geq 1$, the initial segment of $\alpha_i$ is typically shorter than the first half of $\alpha_i$.

    \begin{lemma}
        The image of the initial segment of $\alpha_i$ contains the initial segment of $\alpha_{i+1}$.
    \end{lemma}
    
    \begin{proof}
        By definition, the initial segment of $\alpha_i$ ends with a $k_{o/u}k_{u/o}$ pair with $k\geq 2$. By Corollary \ref{block cancellation}, the image of the initial segment of $\alpha_i$ ends with
        \[g(k_{o/u}k_{u/o})=(k+1)_{o/u}(k+1)_{u/o}.\]To travel from $p$ to $(k+1)_{o/u}(k+1)_{u/o}$, the arc $\alpha_{i+1}$ has to contain at least one instance of the pair  $1_{o/u}2_{o/u}$. Hence, the initial segment of $\alpha_{i+1}$ is contained in the image of the initial segment of $\alpha_i$.  
    \end{proof}

    The following lemma will be our main technical tool in proving Theorem~\ref{main}. 
    
    \begin{lemma}\label{ini_con}
        If $k\geq 0$, then the initial segment of $\alpha_{nk}=g^{nk}(\alpha_0)$ is $\delta_0$. 
    \end{lemma}

\begin{proof}
      We proceed by induction on $k$.  For the base case, $\delta_0$ is the initial segment of $\alpha_0$. We will show that the initial segment of $\alpha_n$ is also $\delta_0$.  To do so, we compute the beginnings of each $\alpha_i$ for $0\leq i\leq n-1$ directly by taking the image of the code for the initial segment of $\alpha_{i-1}$.  The lengths of the images of the initial segments become quite long, and so rather than doing the full computation, we do only enough to include a pair that blocks cancellation. This is sufficient to ensure the portion we do compute will appear in a reduced code for $\alpha_i$.  The  beginnings of the arcs $\alpha_i$ for $n-1\geq i\geq 0$ are  shown in Figure \ref{fig:ini} when $n\geq 3$.
     \begin{figure}[H]
         \centering
         \begin{overpic}[width=1\linewidth]{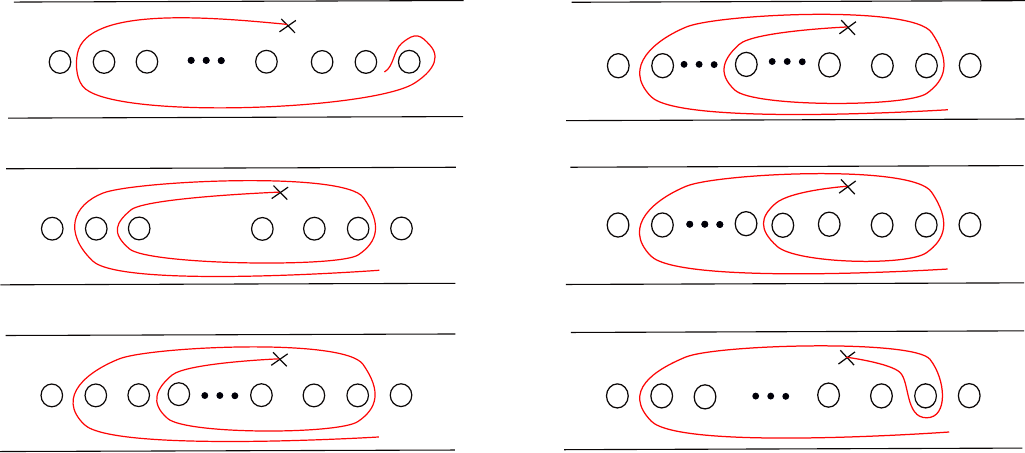}
         \put(0,34){{\color{blue}$i=0$}}
         \put(0,17){{\color{blue}$i=1$}}
         \put(0,1.5){{\color{blue}$i=2$}}
         \put(56,34){{\color{blue}$i$}}
         \put(52,17){{\color{blue}$i=n-4$}}
         \put(52,1.5){{\color{blue}$i=n-1$}}
         \end{overpic}
         \caption{Beginnings of arcs $\alpha_i$ for $0\leq i\leq n-1$ shown on the front of $\Sigma$.}
         \label{fig:ini}
     \end{figure}
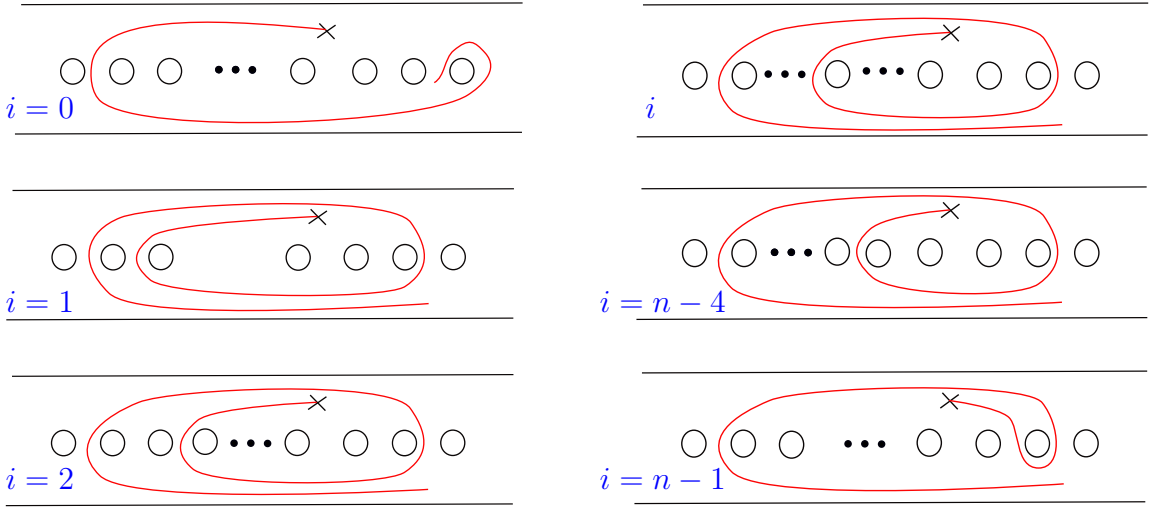

    The initial segment of $\alpha_0$ is $\delta_0$, which ends in a pair that blocks cancellation.  Thus a reduced code for $g(\delta_0)$ appears in a reduced code for $\alpha_{i+1}$.  When $n\geq 3$, direct computation yields $g(\delta_0)$ begins with $P_s0_o(-1)_o\cdots (-n+5)_o(-n+4)_o(-n+4)_u(-n+5)_u\cdots 1_u2_u2_o$.  This ends with $2_u2_o$ and contains exactly one instance of $1_{o/u}2_{o/u}$, and so this is the initial segment for $\alpha_1$.  
    
    Continuing in this manner, we see that for $1\leq i\leq n-3$, the initial segment of $\alpha_i$ has the form $P_s0_o(-1)_o\cdots (-n+4+i)_o(-n+3+i)_o(-n+3+i)_u(-n+4+i)_u\cdots 1_u2_u2_o$. Note that these all end in  $2_u2_o$ and contain exactly one instance of $1_{o/u}2_{o/u}$, so they are the initial segments of $\alpha_i$.

    For $i=n-2$, we compute the image of the initial segment of $\alpha_{n-3}$, which is $P_s0_o0_u 1_u2_u2_o$, to obtain that $g(P_s0_o0_u 1_u2_u2_o)$ begins with $P_s1_u2_u2_o$, the initial segment of $\alpha_{n-2}$.  For $i=n-1$, we compute that $g(P_s1_u2_u2_o)$ begins with $P_s1_o2_u2_o$, which is the initial segment of $\alpha_{n-1}$.  Finally, the image $g(P_s1_o2_u2_o)$ begins with $\delta_0=P_s0_o(-1)_o\cdots (-n+4)_o(-n+3)_o(-n+3)_u(-n+4)_u\cdots 1_u2_u3_u3_o$, which is the initial segment of $\alpha_n$.

    When $n=2$, as above, a reduced code for $g(\delta_0)$ appears in a reduced code for $\alpha_1$ and a reduced code for $g^2(\delta_0)$ appears in a reduced code for $\alpha_2$.  We compute these directly:  $\delta_0=P_s1_u2_u3_u3_o$, $g(\delta_0)=P_s1_o2_u2_o1_oP_oP_u1_u2_u3_u4_u4_o$, and $g^2(\delta_0)=P_s1_u2_u3_u3_o2_u1_u1_o2_u2_o1_oP_oP_u1_u2_u3_u4_u5_u5_o$. We see that $g^2(\delta_0)$, and hence $\alpha_2$, begins with $\delta_0$.

    Thus, for $n\geq 2$, $g^n(\delta_0)$ begins with $\delta_0$.  Since the final pair $3_u3_o$ of $\delta_0$ blocks cancellation, this initial copy of $\delta_0$ persists in the reduced code under further applications of $g$.  It follows inductively that  $\alpha_{kn}$ begins with $\delta_0$ for every $k\geq 0$. 
\end{proof}

    Since $\alpha_t=g^t(\alpha_0)$, $\alpha_{kn+t}=g^t(\alpha_{kn})$, and the last two characters of  $\delta_0$  block cancellation, we immediately obtain the following corollary.

\begin{cor}\label{ini_t}
    For all $0\leq t<n$ and $k\geq 0$, the arcs $\alpha_{kn+t}$ and $\alpha_{t}$ share the segment $g^t(\delta_0)$. 
\end{cor}
    The next lemma shows that $\alpha_{nk+t}$ and $g^n(\alpha_{nk})=\alpha_{n{(k+1)}+t}$ share a long initial subpath for all $t$. 
\begin{lemma}\label{common}
    For all $k\geq 0$ and $0\leq t< n$, the arcs $\alpha_{nk+t}$ and $\alpha_{n{(k+1)}+t}$ agree for the first $\frac{1}{2}\lfloor \ell(\alpha_{nk+t})\rfloor$ characters.
\end{lemma}

\begin{proof}
    We begin with the case $t=0$.  The arcs $\alpha_n$ and $\alpha_0$ share the initial segment $\delta_0$ by Lemma~\ref{ini_con}.  In particular, in the proof of Lemma~\ref{ini_con}, we showed that $g^n(\delta_0)$ begins with $\delta_0$.
    \begin{figure}
        \centering
        \includegraphics[width=1\linewidth]{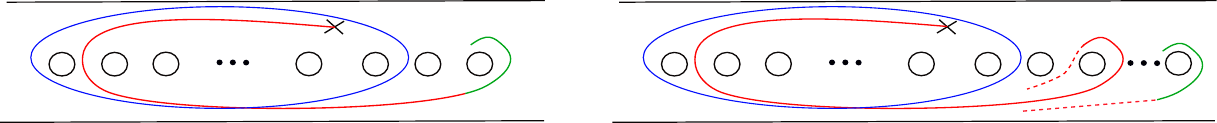}
        \caption{On the left is $\delta_0$, and on the right is an initial and terminal portion of its image under $g^n$.  The image of the green pair on the left is the green pair on the right.  Notice that the initial subsegment of $g^n(\delta_0)$ is $\delta_0$.}
        \label{fig:Ci&n}
    \end{figure}
    Let $\delta_1$ be the segment of $g^n(\delta_0)$ obtained by removing the initial $\delta_0$, so that $g^n(\delta_0)=\delta_0\delta_1$ and $\alpha_n=\delta_0\delta_1\alpha_n'$ for some $\alpha_n'$.  Since $\delta_0$ ends with $3_u3_o$, it follows that $g^n(\delta_0)$, and hence $\delta_1$, ends in $(n+3)_u(n+3)_o$.  Hence Corollary~\ref{block cancellation} implies that there is no cancellation between $g^n(\delta_0)$ and $g^n(\delta_1)$ nor between $g^n(\delta_1)$ and $g^n(\alpha_n')$. Therefore, $g^n(\delta_0\delta_1)=g^n(\delta_0)g^n(\delta_1)=\delta_0\delta_1g^n(\delta_1)$, and so $\alpha_{2n}$ and $\alpha_n$ share the segment $\delta_0\delta_1$. As noted above,  $\ell(\delta_0) = \frac{1}{2}\lfloor \ell(\alpha_0)\rfloor + 1$, because $\alpha_0$ is symmetric about its middle pair $3_u3_o$. Similarly, $\alpha_n$ is symmetric about the final pair $(n+3)_u(n+3)_o$ of $\delta_1$, and so $\ell(\delta_0\delta_1) = \frac{1}{2}\lfloor \ell(\alpha_n)\rfloor + 1$. 
    
    We proceed by induction using the same idea. For $k\geq 2$, let $\delta_k := g^n(\delta_{k-1})$, and assume that $\alpha_{kn}$ and $\alpha_{(k-1)n}$ begin with the segment $\delta_0\delta_1\cdots\delta_{k-1}$, where $\delta_i$ ends with a $k_{o/u}k_{u/o}$ pair  for $k-1\geq i\geq 0$. By Corollary \ref{block cancellation}, there is no cancellation between $g^n(\delta_i)$ and $g^n(\delta_{i+1})$ for $k-2\geq i\geq 0$ nor between $g^n(\delta_{k-1})$ and the rest of $\alpha_{(k+1)n}$. Hence, the code for $\alpha_{(k+1)n}$ begins with the segment 
    \[g^n(\delta_0\delta_1\cdots\delta_k)=g^n(\delta_0)g^n(\delta_1)\cdots g^n(\delta_k) = \delta_0\delta_1\cdots\delta_{k}\delta_{k+1}.\]
    Therefore, $\alpha_{(k+1)n}$ and $\alpha_{kn}$ begin with the segment $\delta_0\delta_1\cdots\delta_{k}$,  where $\delta_i$ ends with a $k_{o/u}k_{u/o}$ pair  for $k\geq i\geq 0$. By the symmetry of $\alpha_{nk}$ about the pair $(nk+3)_u(nk+3)_o$ at the end of $\delta_k$, we have 
    \[\ell(\delta_0\delta_1\cdots\delta_k) = \frac{1}{2}\lfloor \ell(\alpha_{nk})\rfloor + 1.\]
    This completes the induction when $t=0$.

    For $t\neq 0$, the arcs $\alpha_t$ and $\alpha_{kn+t}$ share the segment $g^t(\delta_0)=\delta_0'$ by Corollary \ref{ini_t}. Moreover, the proof of Lemma~\ref{ini_con} shows that $g^n(\delta_0')$ begins with $\delta_0'$.  Let $\delta_1'$ be such that $g^n(\delta_0')=\delta_0'\delta_1'$ and for $k\geq 2$, define $\delta_k':=g^n(\delta_{k-1}')$. An identical argument as above with $\delta_k'$ playing the role of $\delta_k$ shows that the arcs $\alpha_{kn+t}$ and $\alpha_{(k+1)n+t}$ agree on the first $\frac{1}{2}\lfloor\ell(\alpha_{nk+t})\rfloor$ characters.   

    In particular, for every $0\leq t< n$, the above construction gives increasingly long initial subsegments of $\alpha_{kn+t}$, and hence $\ell(\alpha_{kn+t})\to\infty$ as $k\to\infty$.
\end{proof}

    We are now ready to show that the arcs $\alpha_{kn+t}$ converge to geodesic lamination as $k\to\infty$.  In order to apply Saric's criterion for convergence (see Section~\ref{sec:laminations}), we first fix the pants decomposition and the train track $\Theta$ as in \cite[Sections 10.2 \& 10.3]{Abbott_2025}; see Figure \ref{fig:pantsdecomp} for the pants decomposition and Figure \ref{fig:cuffs} for the train track.

    \begin{figure}
    \centering
    \begin{overpic}[width=4in, trim={1.7in 7.55in 1in 1.65in},clip]{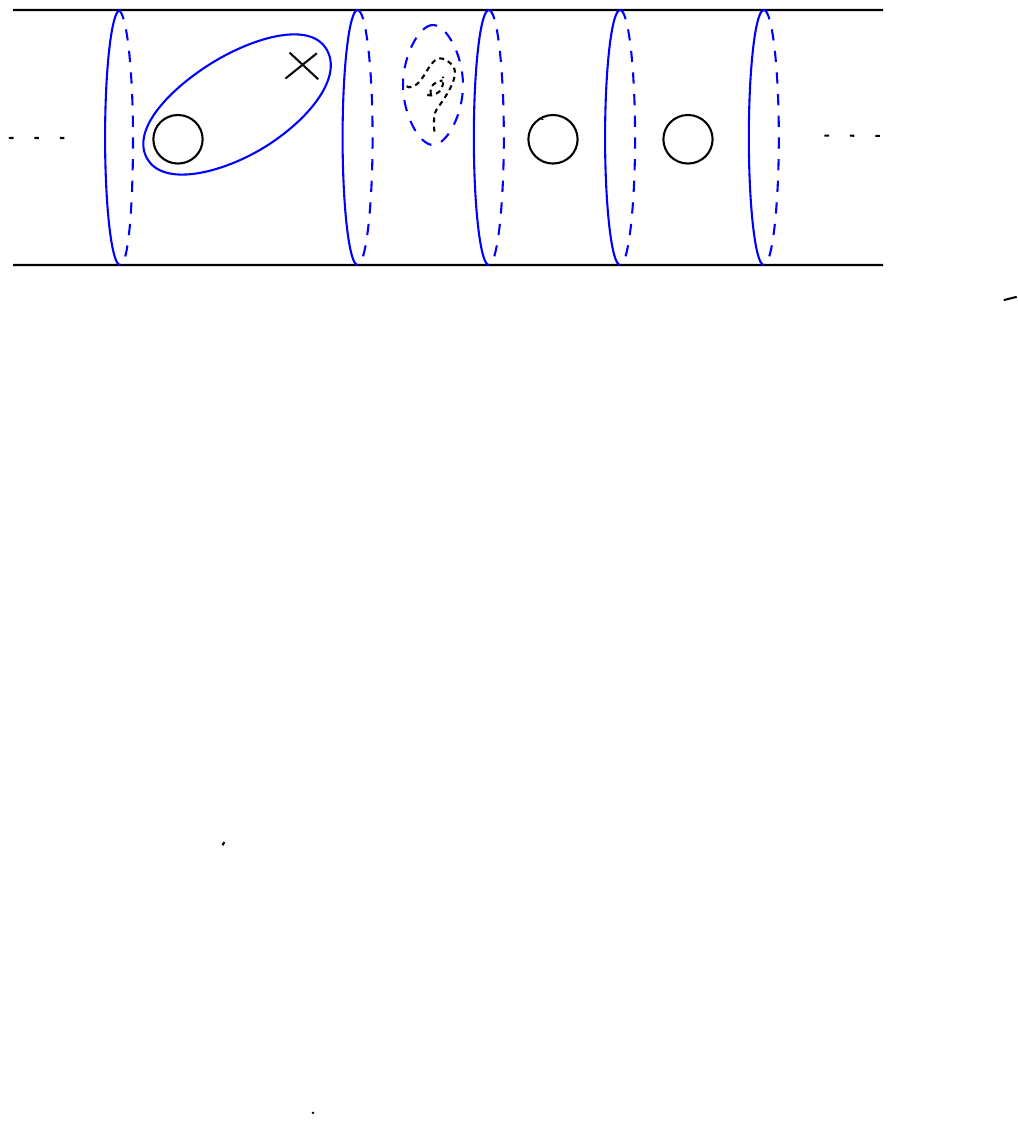}
    \put(22,26){\textcolor{blue}{$\gamma_p$}}
    \put(43,27.5){\textcolor{blue}{$\delta_{0}$}}
    \put(12,-2){$\gamma_{0,\ell}$}
    \put(40,-2){$\gamma_{0,r}$}
    \put(55,-2){$\gamma_{1,\ell}$}
    \put(70,-2){\footnotesize$\gamma_{1,r}=\gamma_{2,\ell}$}
    \put(87,-2){$\gamma_{2,r}$}
    \put(24,15){$B_{0}$}
    \put(65,18.5){$B_{1}$}
    \put(82,12){$B_{2}$}
    \put(27,20){\textcolor{red}{$Q_p$}}
    \put(22,6){\textcolor{red}{$Q_{0}$}}
    \put(45,6){\textcolor{red}{$Q_{0,r}$}}
    \put(62,6){\textcolor{red}{$Q_{1}$}}
    \put(79,6){\textcolor{red}{$Q_{2}$}}
    \put(34,21){$p$}
    \end{overpic}
    \caption{The pants decomposition from \cite{Abbott_2025}.}
    \label{fig:pantsdecomp}
    \end{figure}

    \begin{figure} 
    \centering
    \begin{overpic}[width=4in, trim={1.7in 7.55in 1in 1.65in},clip]{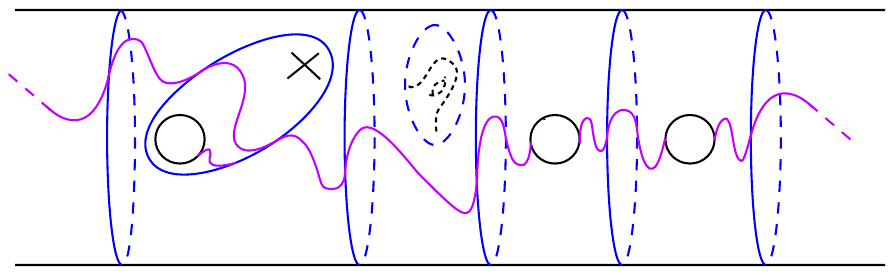}
    \put(22,25){$\scriptstyle\gamma_p$}
    \put(12,-1){$\scriptstyle\gamma_{0,\ell}$}
    \put(40,-1){$\scriptstyle\gamma_{0,r}$}
    \put(55,-1){$\scriptstyle\gamma_{1,\ell}$}
    \put(70,-1){$\scriptstyle\gamma_{1,r}=\gamma_{2,\ell}$}
    \put(87,-1){$\scriptstyle\gamma_{2,r}$}
    \put(29,28){$\scriptstyle P_o$}
    \put(28.75,20){$\scriptstyle P_u$}
    \put(22.5,20){$\scriptstyle0_o$}
    \put(18,9){$\scriptstyle0_u$}
    \put(14,21){$\scriptstyle0_L$}
    \put(31,11){$\scriptstyle0_R$}
    \put(23,17){$\scriptstyle B_{0}$}
    \put(66,12){$\scriptstyle B_{1}$}
    \put(82,12){$\scriptstyle B_{2}$}
    \put(43,8){$\scriptstyle0_{RR}$}
    \put(58.5,10.5){$\scriptstyle1_{L}$}
    \put(74,10){$\scriptstyle2_{L}$}
    \put(66,18.5){$\scriptstyle1_{R}$}
    \put(82,18.5){$\scriptstyle2_{R}$}
    \put(62.5,19){$\scriptstyle1_{o}$}
    \put(62.5,10.5){$\scriptstyle1_{u}$}
    \put(78.5,19){$\scriptstyle2_{o}$}
    \put(78.5,10.5){$\scriptstyle2_{u}$}
    \put(34.3,21.9){$\scriptstyle p$}
    \end{overpic}
    \caption{The train track $\Theta$ from \cite{Abbott_2025}.}
    \label{fig:cuffs}
    \end{figure}
    
    Since our arcs $\alpha_i$ start and end at a puncture on $\Sigma$, we cannot directly apply Proposition~\ref{prop:convtogeod}. Instead, for each $\alpha_i$, we form a corresponding closed curve $C_i$ such that $g(C_i)=C_{i+1}$ following the method in \cite[Section 10.3]{Abbott_2025}. Fix a small disk $D_p$ centered at $p$ that is fixed under $g$ and such that $D_p\cap \alpha_i$ is exactly two segments with endpoints $z_{i,1}$ and $z_{i,2}$ on $\partial D_p$. By a homotopy, we can assume that $z_{i,1}=z_{j,1}$ and $z_{i,2}=z_{j,2}$ for all $i\neq j$. For each $\alpha_i$ we obtain $C_i$ by replacing $\alpha_i\cap D_p$ with the clockwise arc from $z_{i,1}$ to $z_{i,2}$ on $\partial D_p$; see Figure \ref{fig: closed curves}. Note that this choice guarantees that $p$ is in the interior of $C_i$. 

   \begin{figure}[H]
       \centering
       \includegraphics[width=0.42\linewidth]{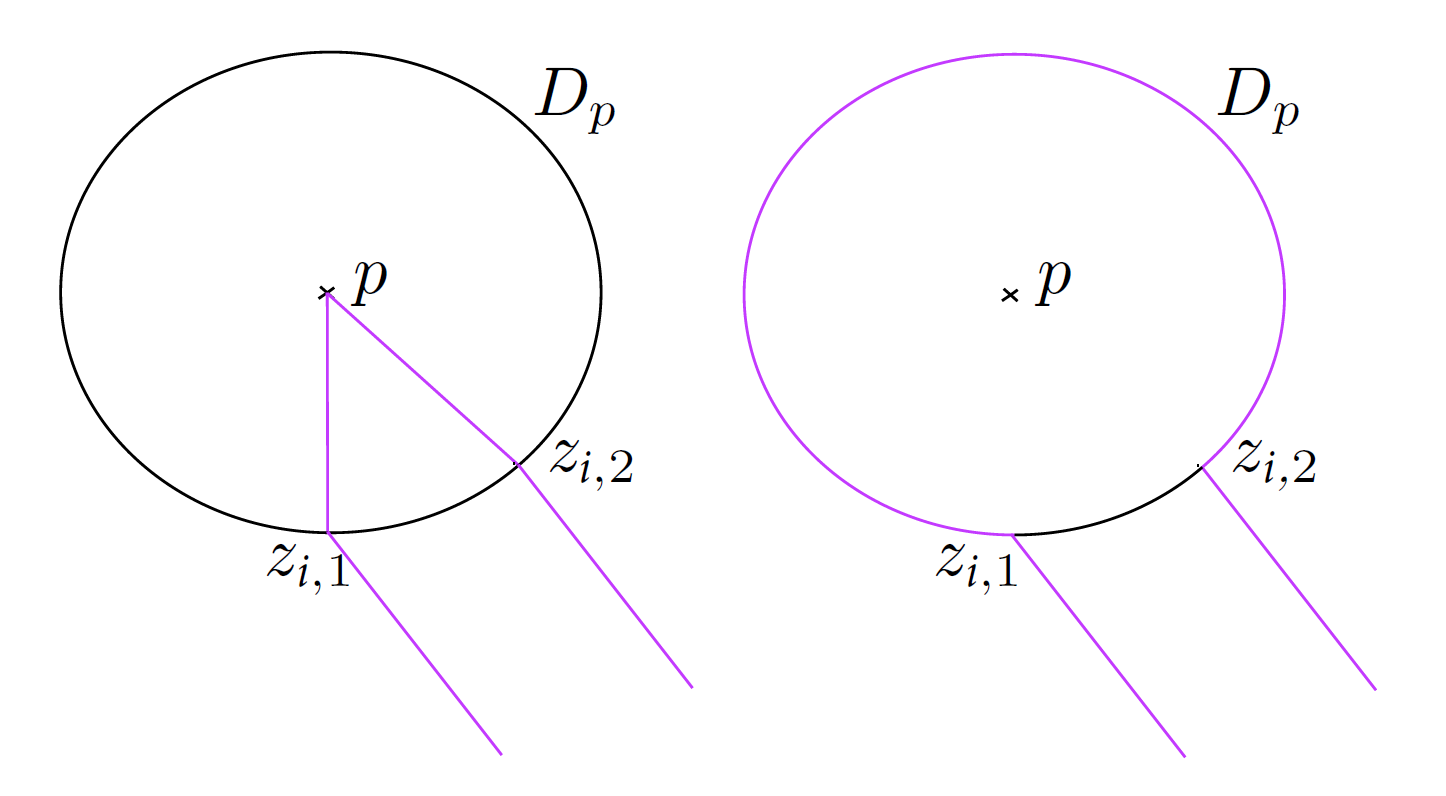}
       \caption{Modifying the arcs $\alpha_i$ (left) to obtain curves $C_i$ (right).}
       \label{fig: closed curves}
   \end{figure}

    We assign a \emph{$\Theta$-code} to $C_i$ as in \cite[Section 10.3]{Abbott_2025}. This is a code particularly suited to the train track $\Theta$. Roughly, each character in the $\Theta$-code for a curve $\alpha$ corresponds to a seam or a cuff in $\Theta$ traversed by $\alpha$. The seams and cuffs of $\Theta$ are labeled with their corresponding character in the $\Theta$-code in Figure~\ref{fig:cuffs}.  It follows from \cite[Lemma~4.4]{Abbott_2025} that the arcs $\alpha_i$ are symmetric, from which it follows that the $\Theta$--code for the curves $C_i$ is symmetric. The following, along with \cite[Theorem~1.2]{Abbott_2025} proves Theorem~\ref{main}.

    \begin{theorem}\label{thm:convergence}
        Let $n\geq 2$.  For each $0\leq t < n$, the closed curves $C_{nk+t}$ converge to a geodesic lamination in the coarse Chabauty topology as $k\rightarrow\infty$. 
    \end{theorem}
    
    \begin{proof}
        Fix $t$.  All $C_i$ share a subarc $\nu$ of $\partial D_p$.  Fix a lift $\tilde \nu$ of $\nu$, and then fix lifts $\widetilde{C_i}$ of $C_i$ containing $\tilde \nu$ for each $i$. We will check that $\widetilde{C_{kn+t}}$ converges to some $\widetilde{\alpha}^t$. Using the $\Theta$-code, we write $C_{nk+t}=a_{1,t}^ka_{2,t}^k\cdots a_{\ell^k_t,t}^k$, where each $a_{i,t}^k$ is the $i$th character in the $\Theta$--code and $\ell^k_t$ is the length of the $\Theta$-code. By the symmetry of $C_{nk+t}$, this code satisfies $a_{j,t}^k=a_{\ell^k_t-j+1,t}^k$ for all $1\leq j\leq \ell_t^k$. The $\Theta$-code for $\widetilde{C_{kn+t}}$ is then the bi-infinite periodic extension $(\cdots ,b_{-1,t}^k,b_{0,t}^k,b_{1,t}^k,b_{2,t}^k,\cdots)$ of this code. After choosing indices so one period occupies the indices $1,\dots, \ell^k_t$, define $b^k_{i,t}=a^k_{i,t}$ for $1\leq i\leq \ell_t^k$ and $b_{i+\ell_k}^k=b_i^k$ for all $i\in \mathbb Z$.  The symmetry of the finite code implies that $b_{-j+1,t}^k=b_j^k$ for all $j\in \mathbb Z$.

    By \cite[Lemma 10.6]{Abbott_2025}, $\ell(\alpha_{nk+t})\leq \ell^k_t \leq 5\ell(\alpha_{nk+t})$. By Lemma \ref{common}, $\alpha_{nk+t}$ and $\alpha_{n(k+1)+t}$ share the first $\frac{1}{2}\lfloor \ell(\alpha_{nk+t})\rfloor$ characters. This means that $C_{nk+t}$ and $C_{n(k+1)+t}$ share at least $\lfloor \frac{1}{10}\ell(\alpha_{nk+t})\rfloor$ initial characters in their $\Theta$-code. We consider the following bi-infinite path
    \begin{equation}\label{eqn:defoftildealphat}
    \widetilde{\alpha}^t = (\cdots c_{-1,t},c_{0,t},c_{1,t},c_{2,t},\cdots),
    \end{equation}
    where for each $k\geq 0$ and $1\leq j\leq \lfloor \frac{1}{10}\ell(\alpha_{nk+t})\rfloor$, we define $c_{j,t}:=b_{j,t}^k$ and $c_{-j+1,t}:=b^k_{-j+1,t}=b_{j,t}^k$. This definition is consistent, since for $1\leq j\leq \lfloor \frac{1}{10}\ell(\alpha_{nk+t})\rfloor$, we have $b_{j,t}^k=b_{j,t}^{k+1}$ for all $k\geq 0$. By construction, $\widetilde{\alpha}^t$ agrees with $C_{nk+t}$ on the characters $c_{j,t}$ with $0< j\leq \lfloor\frac{1}{10}\ell(\alpha_{nk+t})\rfloor$ and $-\lfloor\frac{1}{10}\ell(\alpha_{nk+t})\rfloor+1\leq j\leq 0$ for each $k\geq 0$.

    Equip the space of indexed bi-infinite edge paths in $\tilde \Theta$  with the local topology, in which $\eta_i\to\eta$ if, for every $R\in \mathbb N$, the restrictions of $\eta_i$ and $\eta$ to $[-R,R]$ agree for all sufficiently large $i$.  Let $\mathcal O_t$ consist of the images of $\tilde \alpha^t$ under all deck transformations, together with all reindexings of these bi-infinite paths. Define
    \[
    \Gamma_t:=\overline{\mathcal O_t}.
    \]
    We will show that Proposition~\ref{prop:biinfedgepaths}(i) and (ii) are satisfied by $\Gamma_t$.  For (i), first suppose that $h_1\tilde \alpha^t$ crosses $h_2\tilde\alpha^t$.  By applying $h_1^{-1}$, we may assume that there exists $h$ such that $\tilde \alpha^t$ crosses $h\tilde\alpha^t$.  Since there are periodic edge paths associated to $\tilde C_{nk+t}$ which converge to $\tilde\alpha^t$ as $k\to\infty$, we also have $h\tilde \alpha_{nk+t}\to h \tilde \alpha^t$.  Thus $\tilde \alpha_{nk+t}$ and $h\tilde \alpha_{nk+t}$ cross for sufficiently large $k$.  But these are two lifts of the same simple closed geodesic $C_{nk+t}$, which is a contradiction.  Reindexing cannot create transverse crossings, and so no two elements of $\mathcal O_t$ cross.  Crossing is witnessed by finite subpaths.  Thus, if two elements in $\Gamma_t=\overline{\mathcal O_t}$ cross, then sufficiently close elements of $\mathcal O_t$ must cross as well.  This contradicts the fact that no two elements of $\mathcal O_t$ cross.

    For (ii), suppose $\tilde\beta$ is a bi-infinite edge path in $\tilde \Theta$ such that every finite edge path is contained in some element of $\Gamma_t$.  For each $i$, there exists an element $\eta_i\in \Gamma_t$ which contains $\tilde \beta_{[-i,i]}$.  Since $\Gamma_t$ is closed under re-indexing, we may reindex the sequence $(\eta_i)$ so that $\eta_i|_{[-i,i]}=\tilde\beta|_{[-i,i]}$.   Thus  $\tilde\beta =\lim_{i\to\infty} \eta_i$.  Since $\Gamma_t$ is closed, $\tilde\beta\in \Gamma_t$.  Thus (ii) holds.  By Proposition~\ref{prop:biinfedgepaths}, there is a lamination $L_t$ on $\Sigma$ associated to $\Gamma_t$.

    In fact, if $\widetilde{\gamma}^t$ is the geodesic weakly carried by $\tilde\Theta$ corresponding to $\tilde\alpha^t$ and $\gamma^t$ is the image of $\tilde\gamma^t$ on $\Sigma$, then the  lamination $L_t$ is simply the closure $\overline{\gamma^t}$.  To see this, since $\tilde\alpha^t\in \mathcal O_t\subseteq \Gamma_t$, the corresponding geodesic is a leaf of $L_t$. Because $L_t$ is closed, $\overline\gamma^t\subseteq L_t$. For the reverse inclusion, let $\lambda$ be a leaf of $L_t$, and choose a lift $\tilde\lambda=G(\eta)$ with $\eta\in \Gamma_t$.  Since $\Gamma_t=\overline{\mathcal O_t}$, there exist $\eta_i\in \mathcal O_t$ with $\eta_i\to\eta$ in the local topology. Each $\eta_i$ is a deck translate of a reindexing of $\tilde\alpha^t$, so $G(\eta_i)$ is a deck translate of $\tilde \gamma^t$.  Hence its projection to $\Sigma$ is exactly $\gamma^t$.  Since $\eta_i\to\eta$, it follows from Proposition~\ref{prop:convtogeod} that $G(\eta_i)\to G(\eta)=\tilde \lambda$.  Since every $G(\eta_i)$ projects to $\gamma^t$, every point of the limit $\lambda$ lies in $\overline{\gamma^t}$.  Thus, $\lambda\subseteq \overline{\gamma^t}$. Since this holds for every leaf $\lambda\in L_t$, we have $L_t\subseteq \overline{\gamma^t}$. Thus $L_t=\overline{\gamma^t}$.

    Let $\widetilde{\gamma_{nk+t}}$ be the geodesic weakly carried by $\Tilde{\Theta}$ corresponding to $\widetilde{C_{nk+t}}$. We now check that $\lim_{k\rightarrow\infty}\widetilde{\gamma_{nk+t}}=\widetilde{\gamma}^t$. For any finite subpath $T\subset \widetilde{\alpha}^t$, we can find $h\geq 1$ such that $T\subset c_{-h,t}c_{-h+1,t}\cdots c_{h,t}$. Then for $K$ such that $ \lfloor \frac{1}{10}\ell(\alpha_{nK+t})\rfloor > h$, we have $T\subset \widetilde{C_{nk+t}}$ for all $k>K$. By Proposition~\ref{prop:convtogeod}, it follows that $\lim_{k\rightarrow\infty}\widetilde{\gamma_{nk+t}}=\widetilde{\gamma}^t$. 
    
    While $\widetilde{\gamma}^t$ is a complete geodesic in $\mathbb H^2$, its image $\gamma^t$ on $\Sigma$ may not be closed.  As shown above, the associated geodesic lamination on $\Sigma$ is $L_t=\overline{\gamma^t}$.  We show that $C_{nk+t}$ coarse Chabauty converges to the geodesic lamination $L_t$.  Let $\lambda$ be a leaf of $L_t$, and choose a lift $\tilde \lambda$.  By the correspondence between $L_t$ and $\Gamma_t$, there is some $\eta\in \Gamma_t$ such that $G(\eta)=\tilde\lambda$.  Since $\Gamma_t=\overline{\mathcal O_t}$, there exists a sequence $\eta_i\in \mathcal O_t$ with $\eta_i\to \eta$.  Each $\eta_i$ is the deck translate of a reindexing of $\tilde\alpha^t$.  Thus, for each fixed $i$, applying the corresponding deck transformation and reindexing to $\tilde C_{nk+t}$ gives a sequence of bi-infinite edge paths converging to $\eta_i$, and hence a sequence of lifts $\tilde C_{nk+t}$ whose geodesic representatives converge to $G(\eta_i)$.  For each $i$, choose $k_i$ sufficiently large so that the corresponding lifts of $C_{nk+t}$ agree with $\eta_i$ on the interval $[-i,i]$.  We can choose the sequence $k_i$ to be strictly increasing.  Since $\eta_i\to\eta$, these chosen lifts have edge paths that converge to $\eta$.  By Proposition~\ref{prop:convtogeod}, their geodesic representatives converge to $G(\eta)=\tilde\lambda$.   Thus $\lambda$ is the geodesic limit of the sequence $C_{nk+t}$. Since $\lambda$ was arbitrary,  $C_{nk+t}$  converges to $L_t$ in the coarse Chabauty topology.    
\end{proof}

\begin{figure}
    \centering
    \includegraphics[width=.75\linewidth]{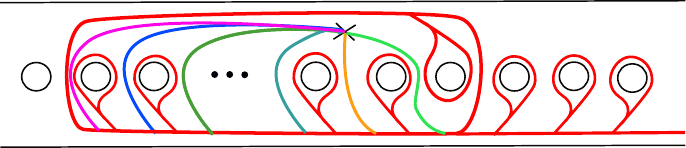}
    \caption{The limiting lamination.  The portion in red is common among all $\beta^t$, while each other color appears in $\beta^t$ for a single $t\in\{0,\dots, n-1\}$.}
    \label{fig:limitinglamination}
\end{figure}

We next prove Theorem~\ref{thm:main2}.  In the proof of Theorem~\ref{main}, we started with a sequence of arcs $g^k(C_0)$ beginning and ending at $p$.  In order to find a limiting lamination, we modified each loop near the puncture $p$ to form a sequence of simple closed curves.  We then showed that this sequence of simple closed curves converges to a lamination.

 Here, elements of $\partial A(\Sigma,p)$ are cliques of rays, which must start at the puncture $p$.  Thus, it is not quite the laminations $\tilde \gamma^t$ that we will show form the attracting limit point of $g$, but rather the limits of the original rays $\lim_{k\to\infty} g^k(C_0)$.  In the proof, we will rely on techniques and results from \cite{Abbott_2022}, which were inspired by \cite{Abbott_2025} and work of Bavard \cite{Bavard}.  In particular, we use the notion of an arc $\beta$ \textit{starting like} an arc $C_i$, which means that the maximal or terminal segment of $C_i$ (not necessarily respectively) has code length at least $\lfloor\frac12\ell_c(C_i)\rfloor-2$.  Roughly, this implies that $\beta$ agrees with $C_i$ for approximately half the length of $C_i$.

 We will use the following lemma about convergence to the boundary in $\mathcal A(\Sigma,p)$.  We say a sequence $(\gamma_i)$ of arcs on $\Sigma$ \textit{cover converge} to $\gamma\subseteq \Sigma$ if there exist lifts $\tilde \gamma_i$ and a lift $\tilde \gamma$ such that $\tilde \gamma_i$ converges to $\tilde \gamma$ in $\mathbb H^2$.

 \begin{lemma}\label{lem:boundarypt}
     Suppose $(x_i)$ is a quasi-geodesic sequence of loops representing a boundary point $P\in \partial \mathcal A(\Sigma,p)$ and a subsequence of $(x_i)$ cover converges to a ray $\lambda$. Then $\lambda$ is high-filling, and $\lambda$ belongs to the clique corresponding to $P$.
 \end{lemma}

 \begin{proof}
     The argument in the second paragraph of the proof of \cite[Theorem~5.1.1]{BavardWalker}, shows that $\lambda$ is a high-filling ray and that the equivalence class of $P(a,\lambda)$ is $P$.  Their description of the boundary implies that $\lambda$ belongs to the clique corresponding to $P$.
 \end{proof}

\begin{proof}[Proof of Theorem~\ref{thm:main2}]
    Fix a mapping class $g_n\in \operatorname{MCG}(\Sigma)$ as in Theorem~\ref{main} with $n\geq 2$.  For simplicity, we denote $g=g_n$.  We will show that the high-filling rays forming the clique in the attracting limit point of $g$ are related to the $n$ geodesic laminations constructed in that theorem.

    We modify the geodesics  $\gamma^t$ by reversing the procedure for building $C_i$ from $\alpha_i$.  Consider the intersection of the small disk $D_p$ around $p$ and $\gamma^t$.  We may choose a representative $\nu^t$ of $\gamma^t$ that contains an arc on $\partial D_p$.  Replace this arc with two segments from its endpoints to the puncture $p$, and call the resulting ray $\beta^t$.  Note that $\beta^t$ and $\nu^t$ agree everywhere outside of $D_p$.  It is clear that we can fix geodesic representatives $\tilde \alpha_{ni+t}$ of lifts of $\alpha_{ni+t}$ to $\mathbb H^2$ and a geodesic representative $\tilde \beta^t$ such that $\tilde \alpha_{ni+t}$ converges to $\tilde \beta^t$ as $i\to\infty$.  
    
    Notice that $\beta^t$ starts like $\alpha_{i}$ for every $i\equiv t\mod n$.  In particular, the sequence $(\alpha_i)_{i\in\mathbb N}$ has a subsequence $(\alpha_i)_{i\equiv t\mod n}$ that cover converges to $\beta^t$.  By Lemma~\ref{lem:boundarypt}, $\beta^t$ is a high filling ray in the clique corresponding to the attracting limit point of $g$.  For $s\neq t$, the rays $\beta^s$ and $\beta^t$ contain distinct initial segments: this follows from Corollary~\ref{ini_t} and the fact that $g^t(\delta_0)$ and $g^s(\delta_0)$ start in distinct sectors; see \cite[Definition~2.9]{Abbott_2025} for the definition of a sector.  Thus this clique contains at least $n$ elements.  

    To show that it contains exactly $n$ elements, we show that if $\delta$ is a ray in the attracting clique disjoint from $\beta^t$, then $\delta$ must be isotopic to $\beta^j$ for some $j$.  Fix $t$, and fix some $i\equiv t\mod n$, so that  $\beta^t$ starts like $\alpha_{i}$.  By \cite[Lemma~4.4]{Abbott_2022}, $g^{-i}(\beta^t)$ starts like $\alpha_0$, and $g^{-i}(\delta)$ is disjoint from $g^{-i}(\beta^t)$.  Here, we are applying  \cite[Lemma~4.4]{Abbott_2022} with $\beta^t$ playing the role of $\delta$, an arc. However, the proof of \cite[Lemma~4.4]{Abbott_2022} only ever uses an initial subsegment of the arc, so it applies equally well to the geodesic $\beta^t$.  Since $g^{-i}(\beta^t)$ starts like $\alpha_0$, it contains the ``beginning" of $\alpha_0$; hence the ray $g^{-i}(\delta)$, being disjoint from $g^{-i}(\beta^t)$, is disjoint from the beginning of $\alpha_0$.  Applying \cite[Proposition~4.10]{Abbott_2022} yields a constant $0\leq k_i\leq 3n +1$ such that $g^{-i+k_i}(\delta)$ starts like $\alpha_1$.  We note that \cite[Proposition~4.10]{Abbott_2022} is stated for an arc  that is disjoint from the beginning of $\alpha_0$, but the proof relies only on the initial portion of the arc, so it applies equally well to the ray $g^{-i}(\delta)$.   Finally, another application of \cite[Lemma~4.4]{Abbott_2022} shows that $\delta$ starts like $\alpha_{i-k_i+1}$. Notice that the constant $k_i$ may depend on $i$. In particular, since $\delta$ starts like $\alpha_{m_i}$, it begins in the same sector as $\alpha_{m_i}$.  The initial segment of $\alpha_{m_i}$ depends only on $m_i\mod n$, by Corollary~\ref{ini_t}.  Thus if $j$ is the sector in which $\delta$ starts, then $m_i:=i-k_i+1 \equiv j \mod n$.
    
    However, since $i=t+n\mathbb N$ was arbitrary and $0\leq k_i\leq 3n+1$, we see that $m_i\to\infty$. Combining this with the fact that $m_i\equiv j\mod n$, we see that after passing to a subsequence we may assume the $m_i$ are strictly increasing.  Thus $(\alpha_{m_i})_{i\geq 0}$ is a subsequence of $(\alpha_{kn_j})_{k\geq 0}$.   Since $\delta$ starts like $\alpha_{m_i}$ for each $i$, and $\lfloor \frac12 \ell_c(\alpha_{m_i})\rfloor-2\to\infty$, the subsequence $(\alpha_{m_i})$ cover converges to $\delta$.  But the sequence $(\alpha_{kn+j})_{k\geq 0}$ cover-converges to $\beta^j$, and so by the uniqueness of cover-convergence limits, we have $\delta=\beta^j$.
    
    We conclude that the clique $\{\beta^t\mid 0\leq t <n\}$ is the attracting limit point of $g$.  In particular, the attracting limit point of $g$ has weight $n$.

    Finally, notice that $\beta^t$ is the cover-convergence limit of $\alpha_{nk+t}$.  Since cover convergence is preserved by homeomorphisms, $g(\beta^t)$ is the cover-convergence limit of 
    \[
    g(\alpha_{nk+t})=\alpha_{nk+(t+1)}.
    \]
    The latter sequence cover converges to $\beta^{t+1}$, where the index $t$ is considered modulo $n$.  Since cover-convergence limits are  unique, $g(\beta^t)=\beta^{t+1}$, concluding the proof when $n\geq 2$.
\end{proof}

 In \cite{Abbott_2025}, the weights of the attracting limit points of the family of mapping classes were not calculated.  However, a similar argument as in the proof of Theorem~\ref{thm:main2} will show that it must have weight one.  While there are no precise analogues of \cite[Lemma~4.4]{Abbott_2022} and \cite[Proposition~4.10]{Abbott_2022} stated in \cite{Abbott_2025}, the analysis of ``highways" from that paper could be used instead.

Finally, we prove Theorem~\ref{thm:main3}.

\begin{proof}[Proof of Theorem~\ref{thm:main3}]
    If $n=1$, then the theorem is immediate from \cite[Theorem~1.2]{Abbott_2025}.  Thus assume $n\geq 2$, and
    let 
    \[
    L=\bigcup_{t=0}^{n-1} L_t.
    \]
    We claim that no leaf of $L_s$ crosses a leaf of $L_t$. Suppose otherwise, and choose crossing  lifts corresponding to paths $\eta_s\in \Gamma_s$ and $\eta_t\in \Gamma_t$.  Since crossing is witnessed by finite subpaths and $\Gamma_i=\overline{\mathcal O_i}$, there are elements of $\mathcal O_s$ and $\mathcal O_t$ which cross.  These correspond to lifts of $\gamma^s$ and $\gamma^t$, so $\gamma^s$ and $\gamma^t$ cross on $\Sigma$.   For each $i$, there is a ray $\nu^i$ with geodesic representative  $\gamma^i$ that differs from $\beta^i$ only inside the disk $D_p$.  Outside of $D_p$, the ray $\beta^s$ and $\beta^t$ are disjoint, while inside $D_p$ the modification replaces their initial and terminal segments with a common arc of $\partial D_p$ and hence introduces no new transverse crossings. Thus the isotopy classes of $\nu^s$ and $\nu^t$ do not cross, and so neither do their geodesic representatives, which is a contradiction. 

    Since each $L_t$ is closed and there are only finitely many of them, their union $L$ is closed.  Therefore $L$ is a geodesic lamination.

    Finally, we show that $g(L_i)=L_{i+1}$, with indices taken modulo $n$.  By Theorem~\ref{thm:main2}, $g(\beta^t)=\beta^{t+1}$.  The modification which recovers $\gamma^t$ from $\beta^t$ is $g$--equivariant, since $D_p$ is $g$--equivariant and the  initial and terminal segments are always replaced with the same common arc of $\partial D_p$.  Thus $g(\gamma_t)$ is isotopic to $\gamma_{t+1}$, and hence the geodesic representative of $g(\gamma^t)$ is $\gamma^{t+1}$.  Since $L_i=\overline{\gamma^i}$, $g(L_t)=L_{t+1}$.  In particular,
    \[
    g(L)=g\left(\bigcup_{t=0}^{n-1}L_t\right)=\bigcup_{t=0}^{n-1}g(L_t)=\bigcup{t=0}^{n-1}L_{t+1}=L.
    \]
\end{proof}

\printbibliography

\end{document}